\documentclass[preprint,12pt]{elsarticle}

\usepackage{amssymb}
\usepackage{amsthm}

\usepackage{amsmath}

\usepackage{bm}

\usepackage{dsfont}

\newtheorem{theorem}{Theorem}[section]
\newtheorem{lemma}[theorem]{Lemma}

\newtheorem{corollary}[theorem]{Corollary}

\theoremstyle{definition}

\theoremstyle{remark}
\newtheorem{remark}[theorem]{Remark}

\makeatletter
\def\@author#1{\g@addto@macro\elsauthors{\normalsize%
    \def\baselinestretch{1}%
    \upshape\authorsep#1\unskip\textsuperscript{%
      \ifx\@fnmark\@empty\else\unskip\sep\@fnmark\let\sep=,\fi
      \ifx\@corref\@empty\else\unskip\sep\@corref\let\sep=,\fi
      }%
    \def\authorsep{\unskip,\space}%
    \global\let\@fnmark\@empty
    \global\let\@corref\@empty  
    \global\let\sep\@empty}%
    \@eadauthor={#1}
}
\makeatother

\journal{}

\begin{document}

\begin{frontmatter}



\title{Convolution identities for scale transformations of Appell polynomials}


\author{Jos\'{e} A. Adell\corref{a1}\fnref{thanks}}
\ead{adell@unizar.es}

\author{Alberto Lekuona\fnref{thanks}}
\ead{lekuona@unizar.es}

\cortext[a1]{Corresponding author.}
\fntext[thanks]{The authors are partially supported by Research Projects DGA (E-64), MTM2015-67006-P, and by FEDER funds.}

\address{Departamento de M\'{e}todos Estad\'{\i}sticos, Facultad de
Ciencias, Universidad de Zaragoza, 50009 Zaragoza (Spain)}

\begin{abstract}
We obtain closed form expressions for convolutions of scale transformations within a certain subset of Appell polynomials. This subset contains the Bernoulli, Apostol-Euler, and Cauchy polynomials, as well as various kinds of their generalizations, among others. We give a unified approach mainly based on a probabilistic generalization of the Stirling numbers of the second kind. Different illustrative examples, including reformulations of convolution identities already known in the literature, are discussed in detail.
\end{abstract}


\end{frontmatter}


\section{Introduction}\label{s1}

In recent years, a lot of attention has been devoted to obtain explicit formulas for higher-order convolutions of the form
\begin{equation}\label{eq1.1}
    \sum_{j_1+\cdots +j_m=n} \binom{n}{j_1,\ldots ,j_m}C(j_1,\ldots ,j_m) A_{j_1}^{(1)}(x_1)\cdots A_{j_m}^{(m)}(x_m),
\end{equation}
where $A^{(k)}(x)=(A_n^{(k)}(x))_{n\geq 0}$ is a sequence of Appell polynomials, $x_k\in \mathds{R}$, $k=1,\ldots ,m$, and $C(j_1,\ldots ,j_m)$ are appropriate constants. Such formulas generalize the classical identity of second order for the Bernoulli polynomials $B(x)=(B_n(x))_{n\geq 0}$ shown by N\"orlund \cite{Nor1922}, i.e.,
\begin{equation}\label{eq1.2}
    \sum_{k=0}^n \binom{n}{k}B_k(x)B_{n-k}(y)=-n(x+y-1)B_{n-1}(x+y)-(n-1)B_n(x+y).
\end{equation}
Since the pioneering work by Dilcher \cite{Dil1996}, many authors have provided explicit expressions for the sums in \eqref{eq1.1} using different methodologies. We mention the papers by Gessel \cite{Ges2005}, Wang \cite{Wan2013}, Agoh and Dilcher \cite{AgoDil2014}, He and Araci \cite{HeAra2014}, Wu and Pan \cite{WuPan2014}, He \cite{He2017}, Dilcher and Vignat \cite{DilVig2016}, and Komatsu and Simsek \cite{KomSim2016}, among others. As it happens in \eqref{eq1.2}, a common feature of the identities for the sums in \eqref{eq1.1} usually found in the literature is that the right-hand side of such identities contains some of the polynomials $A^{(k)}(x)$ themselves.

The aim of this paper is to give a unified approach to obtain closed form expressions for higher-order convolutions of Appell polynomials in the set $\mathcal{R}$ defined below. This approach can be summarized as follows (see Section~\ref{s2} for more precise definitions). Following Ta \cite{Ta2015}, we consider the set $\mathcal{R}$ of Appell polynomials $A(x)=(A_n(x))_{n\geq 0}$ whose generating function is given by
\begin{equation*}
    G(A(x),z)=\dfrac{e^{xz}}{\mathds{E}e^{zY}},
\end{equation*}
for a certain random variable $Y$, where $\mathds{E}$ stands for mathematical expectation. For any $w\in \mathds{R}$ and $A(x)\in \mathcal{R}$, we consider the scale transformation $T_wA(x)=(T_w A_n(x))_{n\geq 0}$ defined by
\begin{equation*}
    T_w A_n(x)=w^n A_n(x/w)=\sum_{k=0}^n \binom{n}{k}w^k A_k(0)x^{n-k},\quad n=0,1,\ldots.
\end{equation*}
In first place, we give closed form expressions for
\begin{equation}\label{eq1.3}
    \begin{split}
       & \sum_{j_1+\cdots +j_m=n} \binom{n}{j_1,\ldots ,j_m} T_{w_1}A_{j_1}^{(1)}(w_1x_1)\cdots  T_{w_m}A_{j_m}^{(m)}(w_mx_m)\\
        & =\sum_{j_1+\cdots +j_m=n} \binom{n}{j_1,\ldots ,j_m} w_1^{j_1}\cdots w_m^{j_m} A_{j_1}^{(1)}(x_1)\cdots  A_{j_m}^{(m)}(x_m)
    \end{split}
\end{equation}
in terms of the moments of the random variables $Y_k$ associated to each $A^{(k)}(x)\in \mathcal{R}$, $k=1,\ldots ,m$ (see Theorem~\ref{th4} in Section~\ref{s3}). In second place, we replace $(w_1,\ldots ,w_m)$ by a random vector $\mathds{W}=(W_1,\ldots ,W_m)$ in \eqref{eq1.3} and then take expectations, so that we obtain closed form expressions for
\begin{equation}\label{eq1.4}
    \sum_{j_1+\cdots +j_m=n} \binom{n}{j_1,\ldots ,j_m} \mathds{E} \left( W_1^{j_1}\cdots W_m^{j_m} \right )A_{j_1}^{(1)}(x_1)\cdots A_{j_m}^{(m)}(x_m).
\end{equation}
This is done in Theorem~\ref{th5} in Section~\ref{s3}, which is the main result of this paper. The comparison between \eqref{eq1.1} and \eqref{eq1.4} reveals the probabilistic meaning of the constant $C(j_1,\ldots ,j_m)$, that is,
\begin{equation*}
    C(j_1,\ldots ,j_m)=\mathds{E} \left ( W_1^{j_1}\cdots W_m^{j_m} \right ).
\end{equation*}

The approach outlined above is general enough for two reasons. First, because the Bernoulli, Apostol-Euler, and Cauchy polynomials, as well as various kinds of their generalizations belong to the set $\mathcal{R}$ (c.f. \cite{Ta2015} and \cite{AdeLek2017}). And second, because each choice of the random vector $\mathds{W}$ in \eqref{eq1.4} leads us to a different type of convolution identity.

The main tool to give explicit expressions for the sums in \eqref{eq1.3} is a probabilistic generalization of the Stirling numbers of the second kind recently introduced in \cite{AdeLek2018} (see also Theorems~\ref{th1} and \ref{th3} in Section~\ref{s3}). The final section is devoted to illustrate Theorem~\ref{th5}. To keep the paper in a moderate size, we restrict our attention to the case in which every $A^{(j)}(x)$ in \eqref{eq1.4} are the Bernoulli polynomials. As a counterpart, we consider different choices of the random vector $\mathds{W}$ and make a comparison with similar results already known in the literature. In this respect, the main difference is that the identities obtained in Section~\ref{s4} are given in terms of the classical Stirling numbers of the second kind, instead of the Bernoulli polynomials or numbers themselves. In other words, we obtain convolution identities easy to compute.

\section{Preliminaries}\label{s2}

In this section, we collect some definitions and properties, already shown in \cite{AdeLek2017} and \cite{AdeLek2018}, which are necessary to state our main results.

Let $\mathds{N}$ be the set of  positive integers and $\mathds{N}_0=\mathds{N}\cup \{0\}$. Unless otherwise specified, we assume from now on that $n\in \mathds{N}_0$, $m\in \mathds{N}$, $x\in \mathds{R}$, and $z\in \mathds{C}$ with $|z|\leq r$, where $r>0$ may change from line to line. Denote by $\mathcal{G}$ the set of all real sequences $\boldsymbol{u}=(u_n)_{n\geq 0}$ such that $u_0\neq 0$ and
\begin{equation}\label{eq2.A}
\sum_{n=0}^\infty |u_n|\dfrac{r^n}{n!}<\infty,
\end{equation}
for some radius $r>0$. If $\boldsymbol{u}\in \mathcal{G}$, we denote its generating function by
\begin{equation*}
G(\boldsymbol{u},z)=\sum_{n=0}^\infty u_n \dfrac{z^n}{n!}.
\end{equation*}
If $\boldsymbol{u},\boldsymbol{v}\in \mathcal{G}$, the binomial convolution of $\boldsymbol{u}$ and $\boldsymbol{v}$, denoted by $\boldsymbol{u}\times \boldsymbol{v}=((u\times v)_n)_{n\geq 0}$, is defined as
\begin{equation}\label{eq2.5}
(u\times v)_n=\sum_{k=0}^n \binom{n}{k}u_k v_{n-k}.
\end{equation}
It turns out that this definition is characterized in terms of generating functions (c.f. \cite[Proposition 2.1]{AdeLek2017}) as
\begin{equation}\label{eq2.6}
G(\boldsymbol{u}\times \boldsymbol{v},z)=G(\boldsymbol{u},z)G(\boldsymbol{v},z).
\end{equation}
In addition (cf. \cite[Corollary 2.2]{AdeLek2017}), $(\mathcal{G},\times)$ is an abelian group with identity element $\boldsymbol{e}=(e_n)_{n\geq 0}$, where $e_0=1$ and $e_n=0$, $n\in \mathds{N}$. Observe that if $\boldsymbol{u}^{(j)}=(u_n^{(j)})_{n\geq 0}\in \mathcal{G}$, $j=1,\ldots ,m$, then
\begin{equation}\label{eq2.7}
    (u^{(1)}\times \cdots \times u^{(m)})_n= \sum_{j_1+\cdots +j_m=n}\binom{n}{j_1,\ldots ,j_m}u_{j_1}^{(1)}\cdots u_{j_m}^{(m)},
\end{equation}
where
\begin{equation*}
    \binom{n}{j_1,\ldots ,j_m}=\dfrac{n!}{j_1!\cdots j_m!},\quad j_1,\ldots ,j_m \in \mathds{N}_0,\quad j_1+\cdots +j_m=n
\end{equation*}
is the multinomial coefficient.

On the other hand, let $A(x)=(A_n(x))_{n\geq 0}$ be a sequence of polynomials such that $A(0)\in \mathcal{G}$. Recall that $A(x)$ is called an Appell sequence if one of the following equivalent conditions is satisfied
\begin{equation}\label{eq2.8}
A'_n(x)=nA_{n-1}(x),\quad n\in \mathds{N},
\end{equation}
\begin{equation}\label{eq2.9}
A_n(x)=\sum_{k=0}^n \binom{n}{k}A_k(0)x^{n-k},
\end{equation}
or
\begin{equation}\label{eq2.10}
G(A(x),z)=G(A(0),z)e^{xz}.
\end{equation}
Denote by $\mathcal{A}$ the set of all Appell sequences. The binomial convolution of $A(x),C(x)\in \mathcal{A}$, denoted by $(A\times C)(x)= ((A\times C)_n(x))_{n\geq 0}$, is defined as (cf. \cite[Section 3]{AdeLek2017})
\begin{equation*}
(A\times C)(x)=A(0)\times C(x)=A(x)\times C(0)=A(0)\times C(0)\times I(x).
\end{equation*}
Equivalently, by
\begin{equation*}
\begin{split}
&(A\times C)_n(x)=\sum_{k=0}^n \binom{n}{k}A_k(0)C_{n-k}(x)=\sum_{k=0}^n \binom{n}{k}C_k(0)A_{n-k}(x)\\
&=\sum_{j_1+j_2+j_3=n}\binom{n}{j_1,j_2,j_3} A_{j_1}(0)C_{j_2}(0)x^{j_3}.
\end{split}
\end{equation*}
As shown in \cite[Theorem 3.1]{AdeLek2017}, $(\mathcal{A},\times)$ is an abelian group with identity element $I(x)=(x^n)_{n\geq 0}$. Also, $(A\times C)(x)$ is characterized by its generating function
\begin{equation}\label{eq2.11}
G((A\times C)(x),z)=G(A(0),z)G(C(0),z)e^{xz}.
\end{equation}

For any $A^{(j)}(x)\in \mathcal{A}$ and $x_j\in \mathds{R}$, $j=1,\ldots ,m$, with $x_1+\cdots +x_m=x$, formula \eqref{eq2.11} implies that
\begin{equation}\label{eq2.12}
    A^{(1)}(x_1)\times \cdots \times A^{(m)}(x_m)=(A^{(1)}\times \cdots \times A^{(m)})(x),
\end{equation}
because both sides in \eqref{eq2.12} have the same generating function.

On the other hand, let $w\in \mathds{R}$ and $A(x)\in \mathcal{A}$. We define the scale transformation $T_wA(x)=(T_wA_n(x))_{n\geq 0}$ as
\begin{equation}\label{eq2.13}
    T_w A_n(x)=w^n A_n(x/w)=\sum_{k=0}^n \binom{n}{k} w^k A_k(0)x^{n-k},\quad w\neq 0,
\end{equation}
and
\begin{equation}\label{eq2.14}
    T_0A_n(x)=A_0(0)x^n.
\end{equation}
As shown in \cite[Proposition~4.1]{AdeLek2017}, $T_w A(x)$ is an Appell sequence characterized by its generating function
\begin{equation}\label{eq2.15}
    G(T_wA(x),z)=G(A(0),wz)e^{xz}.
\end{equation}
In addition, the map $T_w:\mathcal{A}\to \mathcal{A}$ is an isomorphism, whenever $w\neq 0$.

From now on, we will always consider random variables $Y$ fulfilling the integrability condition
\begin{equation}\label{eq2.16}
    \mathds{E}e^{r|Y|}<\infty,
\end{equation}
for some $r>0$. Also, let $(Y_j)_{j\geq 1}$ be a sequence of independent copies of $Y$ and denote by
\begin{equation}\label{eq2.17}
    S_k=Y_1+\cdots +Y_k,\quad k\in \mathds{N}\quad (S_0=0).
\end{equation}
In \cite{AdeLek2018}, we have introduced the Stirling polynomials of the second kind associated to $Y$ as
\begin{equation}\label{eq2.18}
    S_Y(n,r;x)=\dfrac{1}{r!}\sum_{k=0}^r \binom{r}{k}(-1)^{r-k}\mathds{E}(x+S_k)^n,\quad r=0,1,\ldots ,n,
\end{equation}
as well as the Stirling numbers of the second kind associated to $Y$ as
\begin{equation}\label{eq2.19}
    S_Y(n,r)=S_Y(n,r;0),\quad r=0,1,\ldots ,n.
\end{equation}
Equivalently (c.f. \cite[Theorem~3.3]{AdeLek2018}), the polynomials $S_Y(n,r;x)$ are defined via their generating function as
\begin{equation}\label{eq2.20}
    \dfrac{e^{zx}}{r!}(\mathds{E}e^{zY}-1)^r=\sum_{n=r}^\infty \dfrac{S_Y(n,r;x)}{n!}\, z^n,\quad r\in \mathds{N}_0.
\end{equation}
Note that if $Y=1$, we have from \eqref{eq2.19} or \eqref{eq2.20}
\begin{equation*}
    S_1(n,r)=S(n,r),\quad r=0,1,\ldots ,n,
\end{equation*}
$S(n,r)$ being the classical Stirling numbers of the second kind. Explicit expressions for $S_Y(n,r;x)$ for various choices of the random variable $Y$ can be found in \cite[Section~4]{AdeLek2018}.

\section{Main results}\label{s3}

We consider the subset $\mathcal{R}\subseteq \mathcal{A}$ of Appell sequences $A(x)$ whose generating function is given by
\begin{equation}\label{eq3.21}
    G(A(x),z)=\dfrac{e^{xz}}{\mathds{E}e^{zY}},
\end{equation}
for a certain random variable $Y$ fulfilling \eqref{eq2.16}. Note that if $X$ is another random variable satisfying \eqref{eq3.21}, then $X$ and $Y$ have the same law (see, for instance, Billingsley \cite[p. 346]{Bil1995}). For this reason, we say that $A(x)$ has associated random variable $Y$.

It turns out that any Appell sequence $A(x)$ in $\mathcal{R}$ with associated random variable $Y$ can be written in terms of the Stirling polynomials of the second kind associated to $Y$ or in terms of the moments of the random variables $S_k$ defined in \eqref{eq2.17}, as the following result shows.

\begin{theorem}\label{th1}
Let $A(x)\in \mathcal{R}$ with associated random variable $Y$. Then,
\begin{equation}\label{eq3.22}
    \begin{split}
       & A_n(x)=\sum_{r=0}^n (-1)^r r! S_Y(n,r;x)= \sum_{k=0}^n \binom{n+1}{k+1}(-1)^k\mathds{E}(x+S_k)^n \\
        & = \sum_{r=0}^n \binom{n}{r} x^{n-r}\sum_{k=0}^r \binom{r+1}{k+1}(-1)^k\mathds{E}S_k^r.
    \end{split}
\end{equation}
\end{theorem}

\begin{proof}
It follows from assumption \eqref{eq2.16} and the dominated convergence theorem that
\begin{equation*}
    |\mathds{E}e^{zY}-1|<1,\quad |z|\leq s,
\end{equation*}
for some $s>0$. Whenever $|z|\leq s$, we have from \eqref{eq2.20} and \eqref{eq3.21}
\begin{equation*}
    \begin{split}
        & G(A(x),z)=\dfrac{e^{xz}}{1+(\mathds{E}e^{zY}-1)}=\sum_{r=0}^\infty (-1)^r e^{xz}(\mathds{E}e^{zY}-1)^r \\
         & =\sum_{r=0}^\infty (-1)^r r! \sum_{n=r}^\infty \dfrac{S_Y(n,r;x)}{n!}\, z^n=\sum_{n=0}^\infty \dfrac{z^n}{n!}\sum_{r=0}^n (-1)^r r! S(n,r;x),
     \end{split}
\end{equation*}
thus showing the first equality in \eqref{eq3.22}. The second one readily follows from \eqref{eq2.18} and the elementary combinatorial identity
\begin{equation*}
    \sum_{r=k}^n \binom{r}{k}=\binom{n+1}{k+1}.
\end{equation*}
Finally, using \eqref{eq2.9} and the second equality in \eqref{eq3.22}, we obtain
\begin{equation*}
    A_n(x)=\sum_{r=0}^n \binom{n}{r}x^{n-r}A_r(0)=\sum_{r=0}^n \binom{n}{r}x^{n-r}\sum_{k=0}^r \binom{r+1}{k+1}(-1)^k \mathds{E}S_k^r.
\end{equation*}
This shows the third equality in \eqref{eq3.22} and completes the proof.
\end{proof}

\begin{remark}\label{r2}
Assume that $A(x)\in \mathcal{R}$. By Theorem~\ref{th1}, we have $A_0(x)=1$. This implies, by virtue of \eqref{eq2.14}, that $T_0A(x)=I(x)$.
\end{remark}

In the following result, we show that binomial convolutions of scale transformations of Appell sequences in the subset $\mathcal{R}$ also belong to $\mathcal{R}$.

\begin{theorem}\label{th3}
Let $\boldsymbol{w}=(w_1,\ldots ,w_m)\in \mathds{R}^m$ and let $A^{(j)}(x)\in \mathcal{R}$ with associated random variable $Y^{(j)}$, $j=1,\ldots ,m$. Suppose that the random vector $\boldsymbol{Y}=(Y^{(1)},\ldots ,Y^{(m)})$ has mutually independent components. Then, the Appell sequence $(T_{w_1}A^{(1)}\times \cdots \times T_{w_m}A^{(m)})(x)$ belongs to $\mathcal{R}$ with associated random variable $\boldsymbol{w}\cdot \boldsymbol{Y}= w_1Y^{(1)}+\cdots +w_mY^{(m)}$.
\end{theorem}

\begin{proof}
By Remark~\ref{r2}, we can assume without loss of generality that $w_j\neq 0$, $j=1,\ldots ,m$. By assumption,
\begin{equation}\label{eq3.23}
    G(A^{(j)}(x),z)=\dfrac{e^{xz}}{\mathds{E}e^{zY^{(j)}}},\quad |z|\leq r_j,
\end{equation}
for some $r_j>0$, $j=1,\ldots ,m$. Denote by
\begin{equation*}
    r=\min\left ( \dfrac{r_1}{|w_1|},\ldots , \dfrac{r_m}{|w_m|} \right )>0.
\end{equation*}
For $|z|\leq r$, we see from \eqref{eq2.11}, \eqref{eq2.15}, and \eqref{eq3.23} that
\begin{equation*}
    \begin{split}
        & G((T_{w_1}A^{(1)}\times \cdots \times T_{w_m}A^{(m)})(x),z)\\
        & = G(T_{w_1}A^{(1)}(0),z)\cdots G(T_{w_m}A^{(m)}(0),z)e^{xz}\\
         & =G(A^{(1)}(0),zw_1)\cdots G(A^{(m)}(0),zw_m)e^{xz}\\
         &= \dfrac{e^{xz}}{\mathds{E} e^{zw_1 Y^{(1)}}\cdots \mathds{E} e^{zw_m Y^{(m)}}}= \dfrac{e^{xz}}{\mathds{E}e^{z\boldsymbol{w}\cdot \boldsymbol{Y}}},
     \end{split}
\end{equation*}
the last equality because the random vector $\boldsymbol{Y}$ has mutually independent components. Hence, the result follows from \eqref{eq3.21}.
\end{proof}

In the setting of Theorem~\ref{th3}, we use from now on the following notations. Denote by
\begin{equation}\label{eq3.24}
    w_1x_1+\cdots +w_mx_m=x, \quad x_1,\ldots ,x_m\in \mathds{R}.
\end{equation}
Let $(Y_l^{(j)})_{l\geq 1}$ be a sequence of independent copies of $Y^{(j)}$ and assume that the sequences $(Y_l^{(j)})_{l\geq 1}$, $j=1,\ldots ,m$, are mutually independent. Denote by
\begin{equation}\label{eq3.25}
    S_k^{(j)}=Y_1^{(j)}+\cdots + Y_k^{(j)},\quad k\in \mathds{N},\quad S_0^{(j)}=0.
\end{equation}
Observe that for any $k\in \mathds{N}_0$ the sums $S_k^{(j)}$, $j=1,\ldots, m$ are mutually independent.

With these notations, we state our first main result on higher-order convolution identities for scale transformations of Appell sequences.

\begin{theorem}\label{th4}
In the setting of Theorem~\ref{th3}, we have
\begin{equation}\label{eq3.25star}
    \begin{split}
       & \sum_{j_1+\cdots +j_m=n} \binom{n}{j_1,\ldots ,j_m} w_1^{j_1}\cdots w_m^{j_m}A_{j_1}^{(1)}(x_1)\cdots A_{j_m}^{(m)}(x_m)\\
       &=\sum_{r=0}^n (-1)^r r! S_{\boldsymbol{w}\cdot \boldsymbol{Y}}(n,r;x) \\
        & =\sum_{r=0}^n \binom{n}{r} x^{n-r}\sum_{k=0}^r \binom{r+1}{k+1}(-1)^k \mathds{E}\left ( w_1S_k^{(1)}+\cdots + w_mS_k^{(m)} \right )^r.
    \end{split}
\end{equation}
\end{theorem}

\begin{proof}
By \eqref{eq2.12} and \eqref{eq3.24}, we have the basic identity
\begin{equation}\label{eq3.26}
    T_{w_1}A^{(1)}(w_1x_1)\times \cdots \times T_{w_m}A^{(m)}(w_mx_m)=(T_{w_1}A^{(1)}\times \cdots \times T_{w_m}A^{(m)})(x).
\end{equation}
Again by Remark~\ref{r2}, we can assume without loss of generality that $w_j\neq 0$, $j=1,\ldots ,m$ in \eqref{eq3.26}. From \eqref{eq2.7} and \eqref{eq2.13}, we see that
\begin{equation}\label{eq3.27}
    \begin{split}
       & (T_{w_1}A^{(1)}(w_1x_1)\times \cdots \times T_{w_m}A^{(m)}(w_mx_m))_n \\
        & \sum_{j_1+\cdots +j_m=n}\binom{n}{j_1,\ldots ,j_m}T_{w_1}A_{j_1}^{(1)}(w_1x_1)\cdots T_{w_m}A_{j_m}^{(m)}(w_mx_m)\\
        &= \sum_{j_1+\cdots +j_m=n}\binom{n}{j_1, \ldots ,j_m}w_1^{j_1}\cdots w_m^{j_m}A_{j_1}^{(1)}(x_1)\cdots A_{j_m}^{(m)}(x_m).
    \end{split}
\end{equation}
By Theorem~\ref{th3}, the Appell sequence on the right-hand side in \eqref{eq3.26} belongs to $\mathcal{R}$ with associated random variable $\boldsymbol{w}\cdot \boldsymbol{Y}$, and $(w_1Y_l^{(1)}+\cdots +w_mY_l^{(m)})_{l\geq 1}$ is a sequence of independent copies of $\boldsymbol{w}\cdot \boldsymbol{Y}$. We therefore have from \eqref{eq3.25} and Theorem~\ref{th1}
\begin{equation*}
    \begin{split}
        & (T_{w_1}A^{(1)}\times \cdots \times T_{w_m}A^{(m)})_n(x)= \sum_{r=0}^n (-1)^r r! S_{\boldsymbol{w}\cdot \boldsymbol{Y}}(n,r;x) \\
         & =\sum_{r=0}^n \binom{n}{r} x^{n-r} \sum_{k=0}^r \binom{r+1}{k+1}(-1)^k \mathds{E} \left (w_1S_k^{(1)}+\cdots +w_m S_k^{(m)} \right )^r.
     \end{split}
\end{equation*}
This, together with \eqref{eq3.26} and \eqref{eq3.27}, completes the proof.
\end{proof}

The second main result, which generalizes Theorem~\ref{th4}, is the following.

\begin{theorem}\label{th5}
Let $\boldsymbol{W}=(W_1,\ldots ,W_m)$ be a random vector whose components fulfil \eqref{eq2.16}. In the setting of Theorem~\ref{th4}, we have
\begin{equation*}
    \begin{split}
        & \sum_{j_1+\cdots +j_m=n}\binom{n}{j_1,\ldots ,j_m}\mathds{E} \left ( W_1^{j_1}\cdots W_m^{j_m} \right ) A_{j_1}^{(1)}(x_1)\cdots A_{j_m}^{(m)}(x_m) \\
         & =\sum_{r=0}^n \binom{n}{r} \sum_{k=0}^r \binom{r+1}{k+1}(-1)^k \sum_{i_1+\cdots +i_m=r}\binom{r}{i_1,\ldots ,i_m}\times \\
         &\times \mathds{E}\left ( W_1^{i_1}\cdots W_m^{i_m}(x_1W_1+\cdots +x_mW_m)^{n-r}\right )\mathds{E}(S_k^{(1)})^{i_1}\cdots \mathds{E}(S_k^{(m)})^{i_m}.
     \end{split}
\end{equation*}
\end{theorem}

\begin{proof}
By assumption \eqref{eq2.16},
\begin{equation*}
    \mathds{E}e^{r_j|W_j|}<\infty, \quad j=1,\ldots ,m,
\end{equation*}
for some $r_j>0$. This implies that each $W_j$ has finite moments of any order, $j=1,\ldots ,m$. Thus, by H\"older's inequality, all the expectations involving the random vector $\boldsymbol{W}$ in Theorem~\ref{th5} are finite.

From \eqref{eq3.24} and Theorem~\ref{th4}, we have
\begin{equation*}
    \begin{split}
        & \sum_{j_1+\cdots +j_m=n}\binom{n}{j_1,\ldots ,j_m}w_1^{j_1}\cdots w_m^{j_m} A_{j_1}^{(1)}(x_1)\cdots A_{j_m}^{(m)}(x_m) \\
         & =\sum_{r=0}^n \binom{n}{r} \sum_{k=0}^r \binom{r+1}{k+1}(-1)^k\sum_{i_1+\cdots +i_m=r}\binom{r}{i_1,\ldots ,i_m}\times \\
         & \times w_1^{i_1}\cdots w_m^{i_m}(x_1w_1+\cdots +x_mw_m)^{n-r} \mathds{E}(S_k^{(1)})^{i_1}\cdots \mathds{E}(S_k^{(m)})^{i_m}.
     \end{split}
\end{equation*}
Thus, the conclusion follows by replacing $(w_1,\ldots ,w_m)$ by the random vector $\boldsymbol{W}=(W_1,\ldots ,W_m)$ and then taking expectations.
\end{proof}

We emphasize that the random vector $\boldsymbol{W}$ in Theorem~\ref{th5} has not necessarily independent components. In fact, we will consider in Corollary~\ref{c4.3} in Section~\ref{s4} a random vector $\boldsymbol{W}$ such that $W_1+\cdots +W_m=1$.

\section{Examples}\label{s4}

As said in the Introduction, Theorems~\ref{th4} and \ref{th5} can be applied when $A^{(j)}(x)$ are the Bernoulli, Apostol-Euler, and Cauchy polynomials, among many others. Here, we will restrict our attention to the case in which every $A^{(j)}(x)$ are the classical Bernoulli polynomials. As a counterpart, we will consider different choices of the random vector $\boldsymbol{W}$.

In this section, $Y$ is a random variable having the uniform distribution on $[0,1]$, $(Y_j)_{j\geq 1}$ is a sequence of independent copies of $Y$ and
\begin{equation}\label{eq4.29}
    S_k=Y_1+\cdots +Y_k,\quad k\in \mathds{N},\qquad S_0=0.
\end{equation}
Recall that the Bernoulli polynomials $B(x)=(B_n(x))_{n\geq 0}$ are defined via their generating function as
\begin{equation}\label{eq4.30}
    G(B(x),z)=\dfrac{ze^{xz}}{e^z-1}=\dfrac{e^{xz}}{\mathds{E}e^{zY}},
\end{equation}
where the last equality in \eqref{eq4.30} was already noticed by Ta \cite{Ta2015}. On the other hand, Sun \cite{Sun2007} (see also \cite[formula~(38)]{AdeLek2017}) showed the following probabilistic representation for the classical Stirling numbers of the second kind $S(n,k)$
\begin{equation}\label{eq4.31}
    S(n,k)=\binom{n}{k}\mathds{E}S_k^{n-k},\quad k=0,1,\ldots ,n.
\end{equation}
These numbers play a crucial role in any convolution identity referring to the Bernoulli polynomials, as shown in the following result. In this respect, for any random vector $\boldsymbol{W}=(W_1,\ldots ,W_m)$ whose components fulfil \eqref{eq2.16}, we denote by
\begin{equation}\label{eq4.31.1}
    C(j_1,\ldots ,j_m)=\mathds{E}\left ( W_1^{j_1}\cdots W_m^{j_m}\right ),\quad j_1,\ldots ,j_m\in \mathds{N}_0,
\end{equation}
as well as
\begin{equation}\label{eq4.31.2}
    D(i_1,\ldots ,i_m;\boldsymbol{x})=\mathds{E} \left (W_1^{i_1}\ldots W_m^{i_m}\left ( x_1W_1+\cdots +x_mW_m\right )^{n-r}\right ),
\end{equation}
where $i_\nu\in \mathds{N}_0$ and $x_\nu\in \mathds{R}$, $\nu=1,\ldots ,m$. With the preceding notations, we state the following.

\begin{corollary}\label{c4.1}
We have
\begin{equation*}
    \begin{split}
        & \sum_{j_1+\cdots +j_m=n}\binom{n}{j_1,\ldots ,j_m}C(j_1,\ldots ,j_m)\prod_{\nu=1}^m B_{j_\nu}(x_\nu) \\
         & =\sum_{r=0}^n \binom{n}{r}\sum_{k=0}^r \binom{r+1}{k+1}(-1)^k \sum_{i_1+\cdots +i_m=r}\binom{r}{i_1,\ldots ,i_m}D(i_1,\ldots ,i_m;\boldsymbol{x})\times\\
        &\times \prod_{\nu=1}^m\dfrac{S(k+i_\nu,k)}{\binom{k+i_\nu}{k}}.
     \end{split}
\end{equation*}
\end{corollary}

\begin{proof}
In view of \eqref{eq4.31.1} and \eqref{eq4.31.2}, it suffices to apply Theorem~\ref{th5} to the case $A^{(j)}(x)=B(x)$, $j=1,\ldots ,m$. Note that, in such a case, the random sums $S_k^{(j)}$, $j=1,\ldots ,m$ defined in \eqref{eq3.25} have the same law as that of the random sum $S_k$ defined in \eqref{eq4.29}, thus having from \eqref{eq4.31}
\begin{equation*}
    \mathds{E}(S_k^{(j)})^i=\dfrac{S(k+i,k)}{\binom{k+i}{k}},\quad j=1,\ldots ,m,
\end{equation*}
for any $k,i\in \mathds{N}_0$. The proof is complete.
\end{proof}

In contrast with other results found in the literature, we point out that the right-hand side in Corollary~\ref{c4.1} only depends on the random vector $\boldsymbol{W}$ and on the Stirling numbers $S(n,k)$, but not on the Bernoulli polynomials themselves. Different particular cases of Corollary~\ref{c4.1} are obtained for each choice of $\boldsymbol{W}$. The first one is the following.

\begin{corollary}\label{c4.2}
Let $w_j,x_j\in \mathds{R}$, $j=1,\ldots ,m$, as in \eqref{eq3.24}. Then, Corollary~\ref{c4.1} holds for
\begin{equation*}
    C(j_1,\ldots ,j_m)=w_1^{j_1}\cdots w_m^{j_m}, \quad D(i_1,\ldots ,i_m;\boldsymbol{x})=w_1^{i_1}\cdots w_m^{i_m}x^{n-r}.
\end{equation*}
\end{corollary}

\begin{proof}
Recalling \eqref{eq4.31.1} and \eqref{eq4.31.2}, it is enough to choose in Corollary~\ref{c4.1} the deterministic vector $\boldsymbol{W}=(w_1,\ldots ,w_m)$.
\end{proof}

In his classical result, Dilcher \cite{Dil1996} considered the case $w_1=\cdots =w_m=1$ and obtained an identity involving the products $s(m,k)B_j(x)$, where $s(m.k)$ are the Stirling numbers of the first kind. Wang \cite{Wan2013} (see also Chu and Zhou \cite{ChuZho2010}) provided identities when at most two of the numbers $w_j$, $j=1,\ldots ,m$ are different from 1 and the product of Bernoulli polynomials is replaced by a product involving both the Bernoulli and Euler polynomials. In Wang's paper, the resulting formula also depends on such polynomials.

Denote by $\langle x \rangle_n$ the rising factorial, i.e.,
\begin{equation*}
    \langle x \rangle_n=\dfrac{\Gamma(x+n)}{\Gamma(x)},
\end{equation*}
$\Gamma (\cdot)$ being Euler's gamma function. For any $\alpha>0$, denote by $X_\alpha$ a random variable having the gamma density
\begin{equation}\label{eq4.33}
    \rho_\alpha (\theta)=\dfrac{\theta^{\alpha -1}e^{-\theta}}{\Gamma (\alpha)},\quad \theta >0.
\end{equation}
Let $m=2,3,\ldots $ and $\alpha_j>0$, $j=1,\ldots ,m$. Suppose that $(X_{\alpha_j},\ j=1,\ldots ,m)$ are mutually independent random variables such that each $X_{\alpha_j}$ has the gamma density $\rho_{\alpha_j}(\theta)$. We consider the random vector $\boldsymbol{W}=(W_1,\ldots ,W_m)$ defined as
\begin{equation}\label{eq4.34}
    W_j=\dfrac{X_{\alpha_j}}{X_{\alpha_1}+\cdots +X_{\alpha_m}},\quad j=1,\ldots ,m.
\end{equation}
Observe that $W_1+\cdots +W_m=1$. It is well known (cf. Kotz \textit{et al.} \cite{KotBalJoh2000} or Dilcher and Vignat \cite{DilVig2016}) that $\boldsymbol{W}$ has the multivariate Dirichlet distribution and
\begin{equation}\label{eq4.35}
    \mathds{E}\left (W_1^{j_1}\cdots W_m^{j_m}\right )=\dfrac{\langle \alpha_1 \rangle_{j_1}\cdots \langle \alpha_m \rangle_{j_m}}{\langle \alpha_1+\cdots +\alpha_m\rangle_{j_1+\cdots +j_m}},\quad j_l,\ldots ,j_m\in \mathds{N}_0.
\end{equation}
With these ingredients, we enunciate the following result.

\begin{corollary}\label{c4.3}
Let m=2,3,\ldots  and $\alpha_j>0$, $j=1,\ldots ,m$. Then, Corollary~\ref{c4.1} holds for $x_j=x$, $j=1,\ldots ,m$, and
\begin{equation*}
    C(j_1,\ldots ,j_m)=\dfrac{\langle \alpha_1 \rangle_{j_1}\cdots \langle \alpha_m \rangle_{j_m}}{\langle \alpha_1+\cdots +\alpha_m \rangle_n},\ D(i_1,\ldots ,i_m;\boldsymbol{x})=\dfrac{\langle \alpha_1 \rangle_{i_1}\cdots \langle \alpha_m \rangle_{i_m}}{\langle \alpha_1+\cdots +\alpha_m \rangle_r}\, x^{n-r}.
\end{equation*}
\end{corollary}

\begin{proof}
Choose $x_j=x$ and $W_j$ as in \eqref{eq4.34} in Corollary~\ref{c4.1}. From \eqref{eq4.31.1} and \eqref{eq4.35}, we have
\begin{equation*}
    C(j_1,\ldots ,j_m)=\dfrac{\langle \alpha_1 \rangle_{j_1}\cdots \langle \alpha_m \rangle_{j_m}}{\langle \alpha_1+\cdots +\alpha_m \rangle_n},
\end{equation*}
because $j_1+\cdots +j_m=n$. Since $W_1+\cdots +W_m=1$, we have from \eqref{eq4.31.2} and \eqref{eq4.35}
\begin{equation*}
    D(i_1,\ldots ,i_m;\boldsymbol{x})=x^{n-r}\mathds{E}\left (W_1^{i_1}\cdots W_m^{i_m}\right )=\dfrac{\langle \alpha_1 \rangle_{i_1}\cdots \langle \alpha_m \rangle_{i_m}}{\langle \alpha_1+\cdots +\alpha_m \rangle_r}\, x^{n-r},
\end{equation*}
because $i_1+\cdots+i_m=r$. The proof is complete.
\end{proof}

Dilcher and Vignat \cite{DilVig2016} have recently given a similar result to Corollary~\ref{c4.3} in terms of products of the form $B_{l_0}(x)B_{l_1}(0)\cdots B_{l_k}(0)$. This result generalizes Miki's identity (see Miki \cite{Mik1978} and Gessel \cite{Ges2005}).

To show the following result, we will need the following reformulation of the well known Chu-Vandermonde identity (see, for instance, Chang and Xu \cite{ChaXu2011} or Vignat and Moll \cite{VigMol2015} for a probabilistic proof of this identity).

\begin{lemma}\label{l8}
Let $t_1,\ldots ,t_m\in \mathds{R}$ with $t_1+\cdots +t_m=t$. Then,
\begin{equation}\label{eq4.37}
    \sum_{l_1+\cdots +l_m=n}\binom{t_1+l_1}{l_1}\cdots \binom{t_m+l_m}{l_m}=\binom{t+m+n-1}{n}.
\end{equation}
\end{lemma}

\begin{proof}
Use the formula
\begin{equation*}
    \binom{-\beta}{n}=(-1)^n \binom{\beta+n-1}{n},\quad \beta\in \mathds{R}
\end{equation*}
and apply the classical Chu-Vandermonde identity.
\end{proof}

\begin{corollary}\label{c4.5}
Corollary~\ref{c4.1} holds for $x_j=x$, $j=1,\ldots ,m$ and
\begin{equation*}
    C(j_1,\ldots ,j_m)=\prod_{\nu=1}^m j_{\nu}!,\quad D(i_1,\ldots ,i_m;\boldsymbol{x})=x^{n-r}\dfrac{(m+n-1)!}{(m+r-1)!}\,\prod_{\nu=1}^m i_{\nu}!.
\end{equation*}
\end{corollary}

\begin{proof}
Choose in Corollary~\ref{c4.1} a random vector $\boldsymbol{W}=(W_1,\ldots ,W_m)$ whose components are independent and identically distributed random variables, each one having the exponential density $\rho_1(\theta)$ defined in \eqref{eq4.33}. By \eqref{eq4.31.1}, \eqref{eq4.33}, and the independence assumption, we see that
\begin{equation}\label{eq4.38}
    C(j_1,\ldots ,j_m)=\mathds{E}W_1^{j_1}\cdots \mathds{E}W_m^{j_m}=\prod_{\nu=1}^{m}j_{\nu}!.
\end{equation}
Since $x_j=x$, $j=1,\ldots ,m$, we have from \eqref{eq4.31.2}, \eqref{eq4.38}, and Lemma~\ref{l8}
\begin{equation}\label{eq4.39}
    \begin{split}
       & D(i_1,\ldots ,i_m;\boldsymbol{x})=x^{n-r}\sum_{l_1+\cdots +l_m=n-r}\binom{n-r}{l_1,\ldots ,l_m}\mathds{E}W_1^{i_1+l_1}\cdots \mathds{E}W_m^{i_m+l_m} \\
        & =x^{n-r}(n-r)!\prod_{\nu=1}^m i_\nu! \sum_{l_1+\cdots +l_m=n-r}\binom{i_1+l_1}{l_1}\cdots \binom{i_m+l_m}{l_m}\\
        &=x^{n-r}(n-r)!\prod_{\nu=1}^m i_\nu! \binom{m+n-1}{n-r},
    \end{split}
\end{equation}
because $i_1+\cdots +i_m=r$. The conclusion follows from \eqref{eq4.38} and \eqref{eq4.39}.
\end{proof}

Note that Corollary~\ref{c4.5} gives us, after some simple computations, the identity
\begin{equation}\label{eq4.40}
    \begin{split}
       & \sum_{j_1+\cdots +j_m=n}B_{j_1}(x)\cdots B_{j_m}(x) \\
        & =\sum_{r=0}^n \binom{m+n-1}{n-r}x^{n-r}\sum_{k=0}^r \binom{r+1}{k+1}(-1)^k \sum_{i_i+\cdots +i_m=r}\prod_{\nu=1}^m \dfrac{S(k+i_\nu,k)}{\binom{k+i_\nu}{k}}.
    \end{split}
\end{equation}

Agoh and Dilcher \cite{AgoDil2014} considered the left-hand side in \eqref{eq4.40} and showed an identity in terms of products of the form $B_{l_0}(x)B_{l_1}(0)\cdots B_{l_k}(0)$. This result is a generalization of an identity proposed by Matiyasevich (see Agoh \cite{Ago2014} for a proof and further references on Matiyasevich identity).

Let $H(x)=(H_n(x))_{n\geq 0}$ be the Hermite polynomials. Such polynomials can be represented in probabilistic terms (cf. Withers \cite{Wit2000}, Adell and Lekuona \cite{AdeLek2006} or Ta \cite{Ta2015}) as
\begin{equation}\label{eq4.41}
    H_n(x)=\mathds{E}(x+\boldsymbol{i}Z)^n,
\end{equation}
where $\boldsymbol{i}$ is the imaginary unit and $Z$ is a random variable having the standard normal density. Finally, representation \eqref{eq4.41} allows us to give the following convolution identities.

\begin{corollary}\label{c4.6}
Corollary~\ref{c4.1} holds for $x_j=x$, $j=1,\ldots ,m$, and
\begin{equation*}
\begin{split}
    &C(j_1,\ldots ,j_m)=(-\boldsymbol{i})^n \prod_{\nu=1}^{m} H_{j_\nu}(0),\\
    &D(i_1,\ldots ,i_m;\boldsymbol{x})=(-\boldsymbol{i})^n x^{n-r}\sum_{l_1+\cdots +l_m=n-r} \binom{n-r}{l_1,\ldots ,l_m} \prod_{\nu=1}^m H_{i_\nu+l_\nu}(0).
\end{split}
\end{equation*}
\end{corollary}

\begin{proof}
Let $\boldsymbol{W}=(W_1,\ldots ,W_m)$ be a random vector with independent and identically distributed components, each one having the standard normal density. By \eqref{eq4.31.1} and \eqref{eq4.41}, we have
\begin{equation*}
    C(j_1,\ldots ,j_m)=\mathds{E}W_1^{j_1}\cdots \mathds{E}W_m^{j_m}=(-\boldsymbol{i})^n H_{j_1}(0)\cdots H_{j_m}(0),
\end{equation*}
since $j_1+\cdots +j_m=n$. Similarly, from \eqref{eq4.31.2} and \eqref{eq4.41} we have
\begin{equation*}
    \begin{split}
        & D(i_1,\ldots ,i_m;\boldsymbol{x})=x^{n-r}\sum_{l_1+\cdots +l_m=n-r} \binom{n-r}{l_1,\ldots ,l_m} \mathds{E}W_1^{i_1+l_1}\cdots \mathds{E}W_m^{i_m+l_m} \\
         & =(-\boldsymbol{i})^n x^{n-r}\sum_{l_1+\cdots +l_m=n-r} \binom{n-r}{l_1,\ldots ,l_m}H_{i_1+l_1}(0)\cdots H_{i_m+l_m}(0),
     \end{split}
\end{equation*}
because $i_1+\cdots +i_m=r$. Hence, the conclusion follows from Corollary~\ref{c4.1}.
\end{proof}



\section*{References}

\bibliographystyle{elsarticle-num}
\bibliography{mybibfileJMAA}

\end{document}